\documentclass[journal]{IEEEtran}
\usepackage{amsmath}    
\usepackage{graphics,epsfig,subfigure}    
\usepackage{verbatim}   
\usepackage{color}      
\usepackage{subfigure}  
\usepackage{hyperref}   
\usepackage{amssymb}
\usepackage{amsmath}
\usepackage{algorithm}
\usepackage{algorithmic}
\usepackage{amssymb}
\usepackage{epstopdf}
\usepackage{graphicx}
\usepackage{wasysym}
\usepackage{wrapfig}
\usepackage{sidecap}
\usepackage{float}
\usepackage{subfigure}
\newcommand{\expect}[1]{\mathbb{E}\bigg[#1\bigg]}

\newcommand{\script}[1]{{{\cal{#1} }}}

\newcommand{\tsf}[1]{\textsf{#1}}

\newtheorem{lemma}{\textbf{Lemma}}

\newtheorem{theorem}{\textbf{Theorem}}

\begin{document}
\allowdisplaybreaks \abovedisplayskip=.03in \belowdisplayskip=.03in

\title{Codes Can Reduce Queueing Delay in Data Centers}

\author{\large{Longbo Huang, Sameer Pawar, Hao Zhang, Kannan Ramchandran}%
%
%
\thanks{The authors are with the EECS dept at UC Berkeley, Berkeley, CA, 94720, USA. Emails: \{huang, spawar, zhanghao, kannanr\}@eecs.berkeley.edu.}

} \maketitle

\begin{abstract}
In this paper, we quantify how much codes can reduce the data retrieval latency in storage systems. By combining a simple linear code with a novel request scheduling algorithm, which we call Blocking-one Scheduling (BoS), we show analytically that it is possible to reduce data retrieval delay by up to $17\%$ over currently popular replication-based strategies. 
Although in this work we focus on a simplified setting where the storage system stores a single content, the methodology developed can be applied to more general settings with multiple contents. The results also offer insightful guidance to the design of storage systems in data centers and content distribution networks.
\end{abstract}

\section{Introduction}
In today's data centers, one of the most demanding tasks (in terms of latency) is ``disk-read," e.g.,  fetching the data for performing analytics such as MapReduce, or to serve the data to the end consumer. In many cases, this task is greatly complicated by the highly non-uniform data popularity, where the most popular data objects can be accessed ten times more frequently than the bottom third~\cite{scarlett-11}. This skewed demand leads to high contention for read tasks of the most popular data. To meet the demand and reduce data retrieval latency, current systems often introduces data redundancy by replicating many copies of each content, e.g., Hadoop replicates each content three or more times, to make the popular data more available and relieve the hot spots of read contentions, thereby reducing the average request latency.

This motivates us to \emph{investigate the fundamental role of redundancy in improving the system latency}. In particular, we compare two systems, one using codes and the other using simple replication. Heuristically, codes offer more flexibility when retrieving the data from the servers, thus may improve content retrieval latency. In this paper, we use a {\em non-asymptotic analytical approach} to quantify this intuition. 

To make the problem more concrete, consider an abstracted example of retrieving a file in a data center shown in Fig.~\ref{fig:example-intro}. The storage system consists of $4$ storage units, each capable of storing $1$ packet of the desired file that consists of $2$ packets $A$ and $B$. We consider two possible storage strategies 1) each packet is replicated two times; 2) file is encoded using a $(4,2)$ Maximum-Distance-Code (MDS) code. It is easy to see the redundancy factor is $2$ in both cases. There is a common dispatcher that queues and schedules the incoming requests. We assume the request process is Poisson and the service time of the
storage unit is exponentially distributed. 

Quantifying the exact request delay in a coded system is a challenging task. The main difficulty is that \emph{the scheduling algorithm needs to remember which SUs served earlier requests} to ensure that the request of a particular content is always served by distinct SUs. This makes the analysis extremely difficult. Moreover, since data centers cannot afford a redundancy factor of more than a few tens at most, the asymptotic analysis based approach advocated in~\cite{load-balancing-10}~\cite{jiq-2011} are not relavent. 


To tackle this challenge, we develop a novel scheduling algorithm called~\tsf{Blocking-one Scheduling (\tsf{BoS})}. The idea is to block some subsequent content retrieval requests until the head-of-line request is processed. This helps remove the dependency in the scheduling actions, and allows a clean analysis of the delay performance of system using codes. Fortunately, as we will see later, the fraction of throughput loss due to \tsf{BoS} is $O(1/r^2)$,  where $r$ is the redundancy factor. Indeed, even for $r = 2$ in our simple example, the \tsf{BoS} achieves $96\%$ of the maximum system throughput.

Under the above settings, we analytically show that strategy $2$ that uses codes reduces the average request delay, for the considered example, by $7\%$ or more compared to strategy $1$ that uses replication. The intuition behind this is that in the replication system, the requests for packet $A$ or $B$ can be satisfied by only two specific servers, while in the coded system \emph{any} two servers are sufficient to serve the file. Therefore, codes offer more flexibility in data retrieval due to multiplexing gain of available servers.
\begin{figure}[htbp]
\vspace{-.1in}
\centering
\includegraphics[width=3.3in, height=2in]{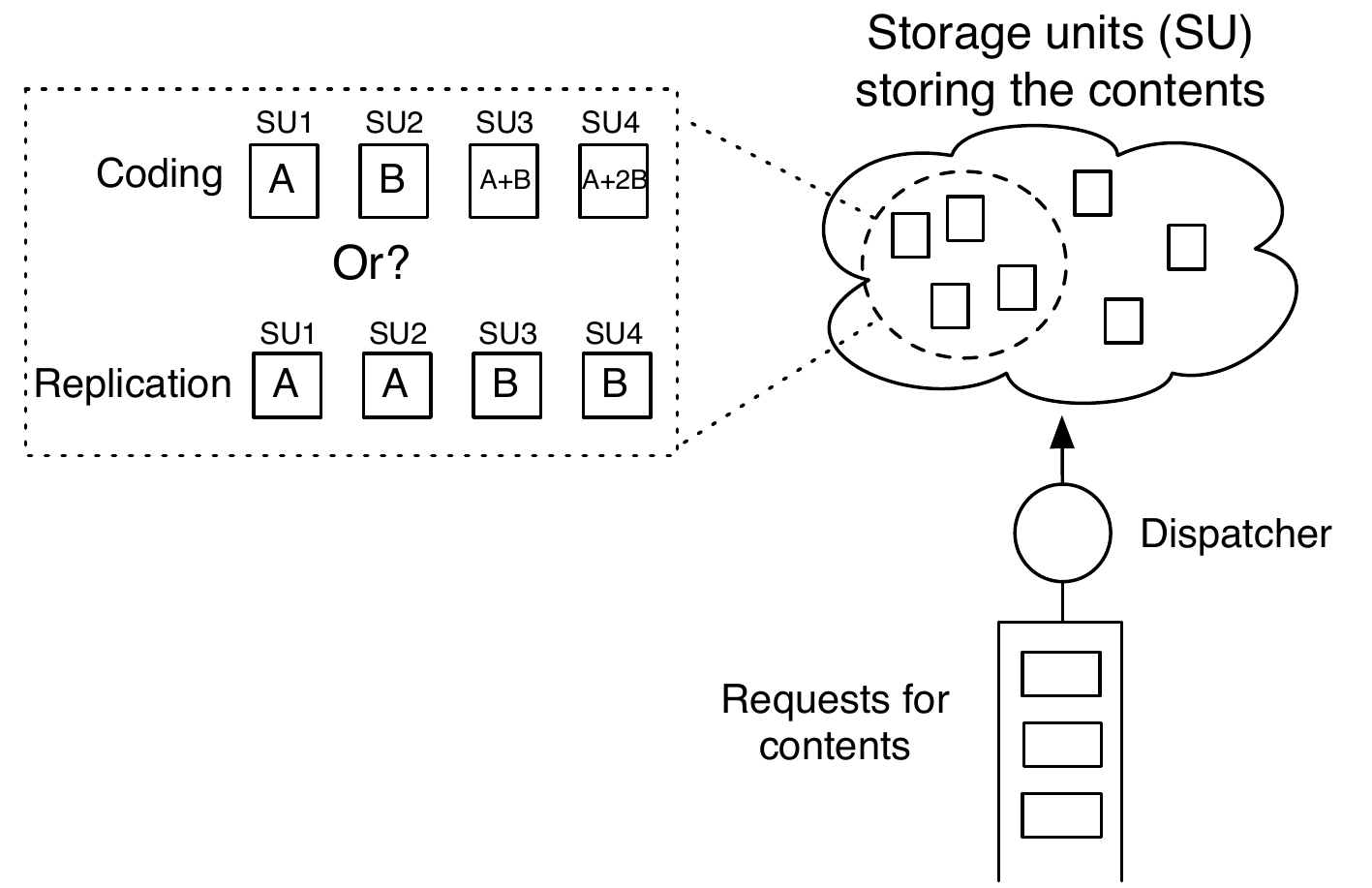}
\vspace{-.1in}
\caption{An example information storage system storing  a file that consists of two packets $A$ and $B$. 
Requests for contents are assigned by a central dispatcher to available storage units. In order to reduce the content retrieval latency, we can either use \emph{replication}, i.e., store the packets $A, B$ twice, or \emph{coding}, i.e., store coded packets of $A$ and $B$. } 
\label{fig:example-intro}
\vspace{-.12in}
\end{figure}

\paragraph*{Related Work}
The authors in \cite{atilla-netcode-tit08} and \cite{yeow-infocm-09} study the throughput-delay gains of network coding in a single hop wireless downlink with unreliable channels. The authors of \cite{code-deadline-lin-10} consider the network coding gains in throughput when packets have hard delay deadlines. The work of \cite{alex-delay-linenet09} studies the gains in delay-throughput when using network coding over a linear network with unreliable links. In \cite{ali-poe-2012}, the authors use network coding to prevent underflow in a multi-media streaming application. While most of the earlier works focus on use of codes to achieve the best effort throughput and or delay over unreliable wireless channels, non-asymptotic and theoretical analysis of queueing delay in coded systems remains largely unaddressed, to the best of our knowledge. 

The paper is organized as follows. In Section \ref{section:model}, we state our system model. In Section \ref{section:schemes}, we present the schemes that will be used in the uncoded and coded systems, and present the Blocking-one Scheduling (\tsf{BoS}) algorithm.  We analyze the performance of \tsf{BoS} in Section \ref{section:analysis}. We then conclude the paper in Section \ref{section:conclusion}.

\section{System Model}\label{section:model}
We consider an information storage system that consists of $n$ homogeneous storage servers, called storage units. Each SU has a storage capacity of one, and is capable of serving each incoming content request in a time that is exponentially distributed with mean $\mu = 1$. The system stores and serves a set of contents, denoted by $\script{C}=\{1, 2, ..., C\}$. Each content is striped into $k$ packets of size one. A total of $k$ SU is needed to store a single content. In practice, for reliability and availability purposes, a content is stored redundantly on multiple servers.

Assume the \emph{content retrieval requests} for each content $c$ form a Poisson arrival process with rate $\lambda_c$. Since a content is striped into $k$ packets that are stored on distinct SUs, every request needs to be served at $k$ servers with distinct packets in order to fully retrieve the desired content. We model this behavior by duplicating each arrived request into $k$ \textit{packet-requests} and by making sure that no two packet-requests are processed by servers with identical packets.

There is a central dispatcher that delegates the incoming requests to the SUs. Upon arrival, the requests are first queued at the central dispatcher, and then sent to the servers once they become available\footnote{Note that such shared queue models have also been widely used to study data center problems, e.g.,\cite{server-farm-with-cost10} \cite{data-size-impact-power11}.}. In such a scenario, we are interested in quantifying the reduction in average content delay  between a coding-based system and a replication-based system.

To illustrate the problem further, consider a simple example depicted in the Fig. \ref{fig:4-server-ex}. The system contains 4 storage units and 1 content. The content is striped into 2 packets $A$ and $B$, and is stored with a redundancy factor of $2$. Each arriving {  content} request {brings into the system two packet-requests, e.g., $3A,3B$.} In Fig. \ref{fig:4-server-ex}(a) we show a system that uses replication to introduce redundancy, where packets $A$ and $B$ are replicated twice. In contrast, Fig. \ref{fig:4-server-ex}(b) shows a system that uses a MDS code of rate $0.5$ to introduce the desired redundancy. The question we aim to answer in this paper is { \emph{by how much can we reduce the average delay of a coding-based system over a replication-based one?}}
\begin{figure}[htbp]
\vspace{-.1in}
\centering
\includegraphics[width=3.3in, height=1.6in]{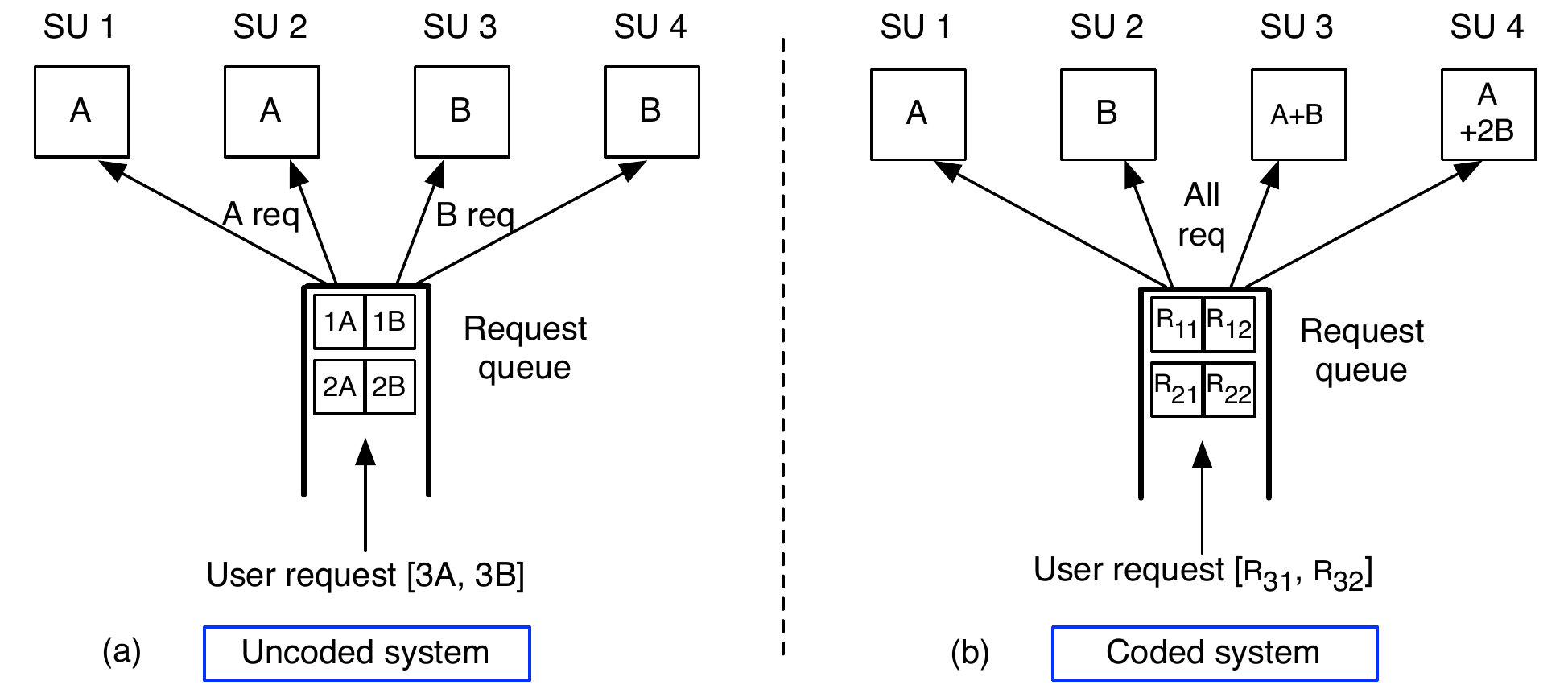}
\vspace{-.1in}
\caption{An example storage system with 4 storage units (SU). The letter inside the SU box corresponds to the packet it stores. Each request arrival brings two packet requests into the system, which must be served by $2$ SUs with distinct packets. A single queue is maintained for all the requests. (a) shows the uncoded system, where the packet $A$ requests are served by SU $1$ and $2$, and the packet $B$ requests are served by SU $3$ and $4$. (b) shows the coded system, where any two distinct SUs can serve the both packet requests $A$ and $B$.}
\label{fig:4-server-ex}
\vspace{-.1in}
\end{figure}

In general, quantifying the delay performance of a system with multiple contents, multiple servers and multiple coded packets that are replicated multiple times can be quite challenging.
{   The main difficulty lies in the fact that the SUs now store coded data, thus the scheduling algorithm has to ensure that the requests for  the same content are not served by a  same SU.
This means that the scheduling algorithm may need to remember at which SUs all the earlier requests are served.
With such a long memory, even defining an appropriate system state  is very hard,  let alone analyzing it. 
}
To make the analysis tractable, we focus on a small-sized problem with a single content, $2$ packets, and $n$ servers. We provide a theoretical upper bound on the average content retrieval delay performance of the coded system, which we show still beats that of a replication-based system with identical amount of SU resources. We believe that the methodology and algorithms developed in this paper will provide useful insights to more general cases, which are part of our ongoing work. 

\section{Storage and scheduling schemes}\label{section:schemes}
In this section, we present the storage and scheduling schemes for both uncoded
and coded systems. We assume that \emph{both systems have identical resources}: each
system hosts a single content that is divided into $k=2$ packets, and has $n=2r, r
\geq 2$ SUs, each capable of storing $1$ packet. Each server can serve requests
with rate of $\mu=1$. Here $r$ is a redundancy factor in the system for both
content availability as well as reliability of the content.

\vspace{-.1in}
\subsection{Uncoded system}
In this case, due to inherent symmetry of arrivals of sub-requests for packets
$A,B$, $r$ SUs  store packet $A$ and the remaining $r$  store packet $B$. Then, whenever  there is
an idle SU that stores packet $A$, the packet $A$ request from the head-of-line request  is assigned to this server. The same happens for  packet $B$ requests.

Under this setting, the uncoded system can be modeled as two $M/M/r$ queueing systems (See Fig. \ref{fig:4-server-ex}). 
\footnote{Note that the two separate systems are indeed not independent, since
the arrivals to the systems happen at the exact same time. However, this does
not affect the analysis for the average request delay. } 
Now denote by $\pi_i$ the steady-state probability that there are $i$ packet requests in an $M/M/r$ system, and denote $\rho\triangleq\frac{\lambda}{r\mu}$. We note that $\{\pi_0,\pi_1,\hdots,\}$ can be computed as follows \cite{data-net-bertsekas}.  
\begin{eqnarray}
\pi_0 &=&\bigg( \sum_{m=0}^{r-1} \frac{(r\rho)^m}{m!} + \frac{(r\rho)^r}{r!}\cdot \frac{1}{1-\rho}\bigg)^{-1}, \\
\pi_m &=&  \frac{(r\rho)^m}{m!}\pi_0, \,\,\,\forall\,\, m\in\{1,\hdots, r\},\\
\pi_{r+m} &=& \rho^m \frac{(r\rho)^r}{r!}\pi_0,\,\,\,\forall\,\,m\geq1.
\end{eqnarray}
And the average packet request delay $d^{\text{uncoded}}_{\text{packet}}$ can be computed by:
\begin{eqnarray}
d^{\text{uncoded}}_{\text{packet}} = \frac{\sum_{n=0}^\infty n\pi_n}{\lambda}. \label{eq:little-law}
\end{eqnarray}
Although the uncoded system admits an easy analysis of the average packet 
request delay, we see that finding the average request delay can be quite
challenging. 
However, as we will see in the coded system, the average request delay can be 
easily computed under our algorithm. 


\vspace{-.1in}
\subsection{Coded system: Blocking-One Scheduling (BoS)}
We now specify our coding and scheduling schemes for the coded system. 
We have $n=2r$ SUs and $k=2$ packets $A,B$. We adopt a simple linear $(n,2)$ MDS
code (any family of MDS code will do) to generate $n$ encoded packets that are
stored at each of the SU. Due to the MDS nature of the code,  any request that is
served at \emph{any two distinct SUs will be able to retrieve the full content}, e.g., see Fig.~\ref{fig:4-server-ex}(b).
%
Under such a coding scheme, in order to minimize average packet request delay, a
simple greedy scheduling strategy would be: queue all the requests in a single
queue. Whenever there is an idle SU that can serve any request in the queue,
assign the request to that SU. Although the suggested greedy scheme seems
simple, it has inherent memory/dependency in scheduling the requests due to the
use of codes. For example, if the first packet request of the $i^{th}$ request, denoted by $R_{i1}$, is
served at SU $j$, then the greedy scheduler has to remember not to assign the other 
packet request $R_{i2}$ to SU $j$ even if it becomes idle. This dependency builds
infinite memory into the system, which makes it very challenging to exactly
analyze the average delay performance of requests in the coded system.

In order to resolve this difficulty,  we propose a novel scheduler called \tsf{Blocking-one Scheduling (BoS)}, for the coded system. The main idea of \tsf{BoS} is to break the memory in the scheduler by blocking
the requests beyond the head-of-line request until both the packet requests of the head-of-line  request are
served.  
The BoS scheduler also corresponds to first-come-first-serve (FCFS).
%
As we will see, the \tsf{BoS}  algorithm not only greatly simplifies the
analysis of the packet request delay, but also allows us to directly calculate 
the average request delay. 
%
\begin{algorithm}
\caption{\tsf{Blocking-one Scheduling  (BoS)}} \label{algorithm:boa}
\begin{algorithmic}[1]
\STATE  At any time $t$, denote the set of idle SUs as
$\script{S}_{\text{Idle}}$, do:
\begin{enumerate}
\item If $\script{S}_{\text{Idle}}\neq\phi$, assign the packet requests from the head-of-line request to an idle SU in $\script{S}_{\text{Idle}}$ as follows: 
\begin{itemize}
\item If  packet request $1$ has not yet been assigned, assign it to the idle SU. 
\item Else assign packet request $2$ if corresponding packet request $1$ was not
served by the idle server.
\end{itemize}

\item If a packet request is assigned to the SU, change the state of SU from idle to busy and remove it from $\script{S}_{\text{Idle}}$, and remove the packet request from queue.
\item Repeat step 1) until no further assignment can be made.
\end{enumerate}

\end{algorithmic}
\end{algorithm}

Note that under \tsf{BoS}, it can happen that there exists an idle SU but no assignment is  made even when the
queue is non-empty. For example, when the free SU has served the packet request $1$ of the head-of-line request and no other SU is idle. In this case, 
the requests beyond the head-of-line request are ``blocked.'' Due to this blocking effect, there
will be a throughput loss due to the lost scheduling opportunity. Fortunately,
as we will see later, the fraction of throughput loss is $O(\frac{1}{r^2})$
that goes to zero as the $r$ increases. Indeed, even for $r=2$, the \tsf{BoS}
achieves $96\%$ of the maximum system throughput.

\section{Analyzing the \tsf{BoS} algorithm}\label{section:analysis}
We now analyze the \tsf{BoS} algorithm by finding the steady-state distribution of the coded system under \tsf{BoS}. The approach works as follows.  We first derive the continuous-time Markov chain that captures the system evolution. Then, we analyze the Markov chain by carefully choosing a set of global balance equations that allow us to compute the steady-state distribution.
\begin{figure*}[cht]
\centering
\includegraphics[width=6in, height=2.2in]{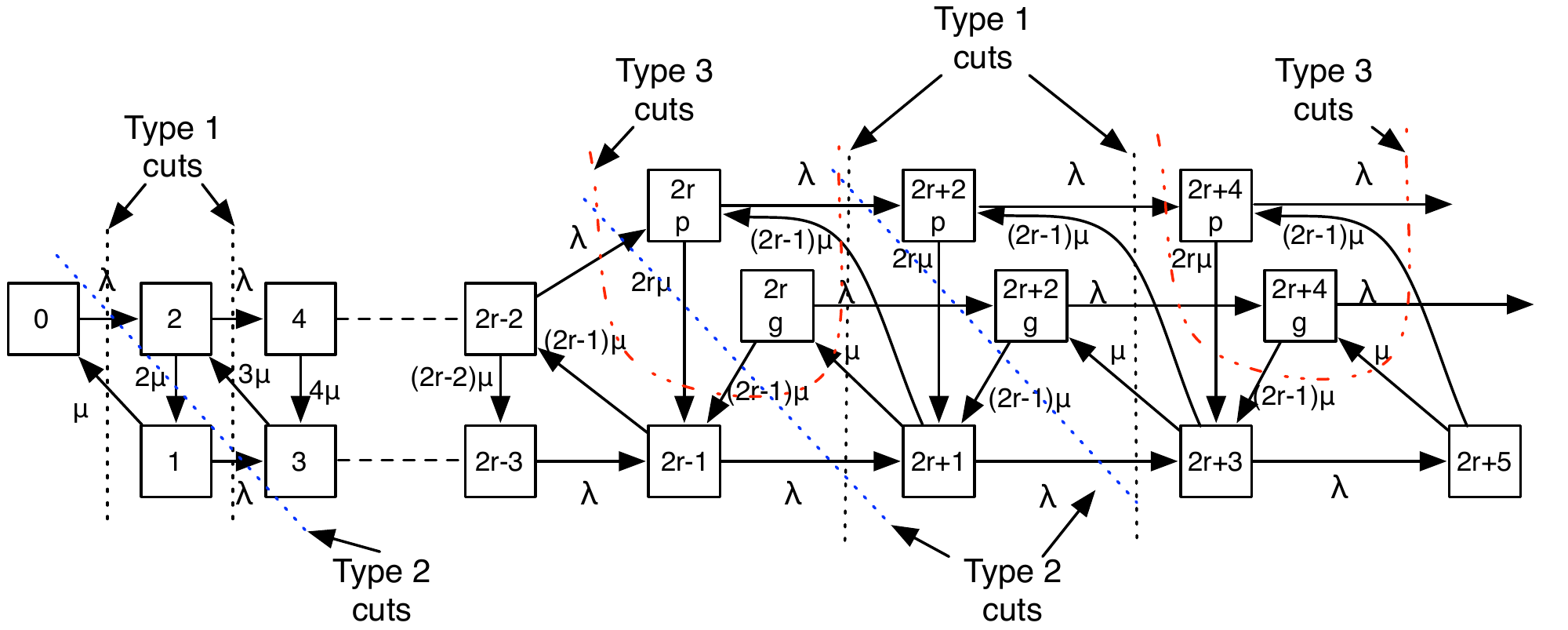}
\caption{The continuous-time Markov chain of the system with $2r$ storage units. The request arrival rate is $\lambda$ and the service rate of each storage unit is $\mu$. 
Each state of the chain represents the number of packet requests in the system. The letters ``p'' and ``g'' represent the perfect and good states. The cuts are used to specify the global balance equations we will use in our analysis. 
} 
\label{fig:mc-system}
\vspace{-.15in}
\end{figure*}

\subsection{The system evolution}
In this section, we present the Markov chain that models the system evolution.
Towards that end, we first take a closer look at the state evolution of the
example system in Fig.\ref{fig:4-server-ex} (b) with 4 SUs.
\begin{figure}[cht]
\vspace{-.1in}
\centering
\includegraphics[width=3.3in, height=1.3in]{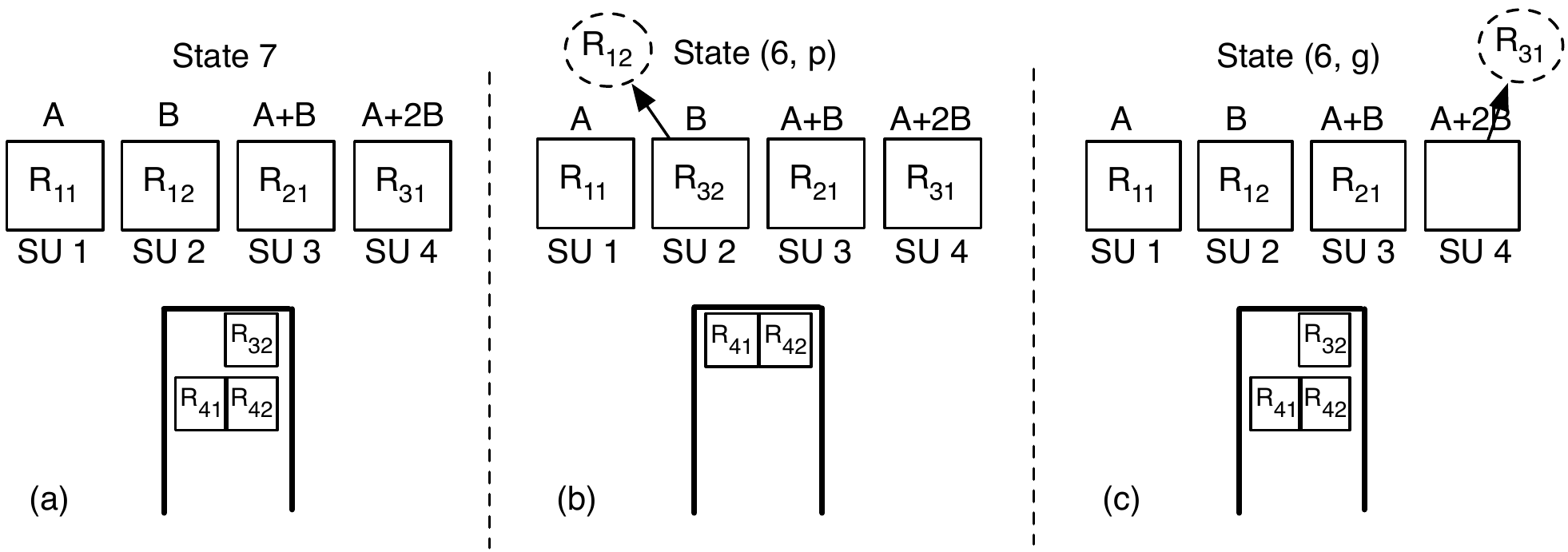}
\vspace{-.1in}
\caption{System evolution of an example with $4$ storage units (SU). (a) shows a state with $7$ packet requests in the system with an aggregate service rate of $4\mu$.  (b) shows the resulting state after the departure of $R_{12}$ from SU $2$, denoted by $(6, p)$, where the system ``renews'' itself and every queued requests can go to any SU.  (c) shows the resulting state after the departure of $R_{31}$ from SU $4$, denoted by $(6,g)$. In this case, one SU is ``wasted'' due to blocking, and the system serves requests with rate $3\mu$.
}
\label{fig:state-explanation}
\end{figure}
Suppose the system is in a  state as shown in Fig. \ref{fig:state-explanation}
(a). In this case, the $4$ SUs are serving  packet requests $R_{11},R_{12},R_{21}$,
and $R_{31}$. There are three more packet requests $R_{32}, R_{41}$ and $R_{42}$ in the queue. Thus, the total
number of packet requests in the system is $7$. We denote the state of the system
by the total number of packet requests in the system i.e., state is $7$.

Now, if SU $2$ completes its service, the packet request $R_{32}$ can be
assigned to SU $2$. This results in a state $(6,p)$ as shown in the Fig.
\ref{fig:state-explanation} (b), which we refer to as a ``perfect-state." This
state is called ``perfect" because once in this state, past evolution is
irrelevant in determining the future scheduling events and the total service
rate is maximum i.e., $4\mu$.  
Note that the system will also enter $(6, p)$ if either SU $1$ or $3$ completes its 
service. Hence, the transition rate from state $7$ into $(6, p)$ is $3\mu$. 
However, if starting from state $7$ but 
SU $4$ completes the service, then we cannot assign request $R_{32}$ to SU $4$,
since $R_{31}$ was served at SU $4$. 
In this case \tsf{BoS} will block the
requests $R_{41},R_{42}$ until $R_{32}$ gets assigned to some other SU. This
results in a not so perfect state, which we denote by $(6,g)$, i.e., good
state with $6$ packet requests in the system (see Fig.\ref{fig:state-explanation}
(c)).  
In this good state, the aggregate service rate is $3\mu$. %
Note that different from $(6, p)$, the system enters $(6, g)$ from state $7$ only when 
SU $4$ finishes its service. Hence, this transition happens only with rate $\mu$. 

Although \tsf{BoS} performs sub-optimally as compared to a system with codes
that maintains an infinite memory and assigns $R_{41}$ to SU 4 when it becomes
idle, it permits analytical treatment of average request delay. 
%
It can be verified that this blocking situation appears only when the system is in a 
state that has more than $2r$ packet requests in the system, and the number of 
packet requests is even.  
Thus, to characterize the system state, we introduce the suffix ``p'' and
``g,'' which stand for ``perfect'' and ``good,'' for all the states with $2r+2m, m\geq0$ packet requests in the system.  
When the number of
packet requests in the system is less than $2r$, we simply use the total number of packet requests
in the system to denote the state of the system.

Now that the system states have been defined, the Markov chain in Fig.
\ref{fig:mc-system} explains the evolution of the system under \tsf{BoS}. The
chain can be understood as follows:
\begin{itemize}
\item With rate $\lambda$, there is an request arrival event,  and two new packet requests are added to the
system. If system was in a good state before arrival it remains in the good
state after arrival and likewise for a perfect state.
\item For any state $<2r$ transition rate for service completion is equal to $\mu$ times the state.
\item For any odd state $>2r$ there are two outgoing transitions for service
completion:
\begin{enumerate}
\item With rate $(2r-1)\mu$, there is transition to an even perfect state.
\item With rate $\mu$, there is transition to an even good state.
\end{enumerate}
\item From an even-perfect state $\geq 2r$, there is a service completion transition to an odd state with rate $2r\mu$.
\item From an even-good state $\geq 2r$, there is a service completion transition to an odd state with rate $(2r-1)\mu$.
\end{itemize}

\subsection{Performance of \tsf{BoS}: Average packet request delay}
Here we present the performance results of \tsf{BoS}. 
To state the theorem, we first define a few notations: 
\begin{eqnarray*}
\hspace{-.3in}&&\eta\triangleq \frac{\lambda}{2r\mu} +\frac{\lambda(2r-1)\mu}{(2r\mu)^2} + \frac{\lambda\mu}{(2r-1)\mu2r\mu}, \\
\hspace{-.3in}&&\gamma_p \triangleq  \frac{2r\mu(\lambda+(2r-1)\mu)}{\mu}  + (\lambda+2r\mu), \\
\hspace{-.3in}&&\gamma_g \triangleq  \frac{-(2r-1)\mu(\lambda+2r\mu)}{2r\mu} - (2r-1) (\lambda +(2r-1)\mu),\nonumber\\
\hspace{-.3in}&&\beta_p \triangleq \frac{\lambda (\lambda + (2r-1)\mu)}{\mu}, \,\,\beta_g \triangleq  \frac{-\lambda(\lambda+2r\mu)}{2r\mu}.
\end{eqnarray*}
Also, for $l\in\{0, ..., 2r-1\}$, define $a_l$ as follows: 
\begin{eqnarray}
a_0=1, a_1 = \frac{\lambda}{\mu}a_0, a_l = \frac{\lambda}{l\mu} (a_{l-1}+a_{l-2}). \label{eq:a-def}
\end{eqnarray}
Now denote by $\pi_0, ..., \pi_{2r-1}, \pi_{2r}^p, \pi_{2r}^g, \pi_{2r+1}, ...$ the stationary distribution of the Markov chain in Fig. \ref{fig:mc-system}, where superscripts  ``p'' and ``g'' stand for perfect and good states. We now have the following theorem. 


%
\begin{theorem}\label{theorem:BoS-per}
Under \tsf{BoS}, we have the following:
\begin{enumerate}
\item[(a)] The maximum rate the coded system can support is:
\begin{eqnarray}
0\leq \lambda < r\mu \bigg(1- \frac{1}{8r^2 - 4r + 1}\bigg).
\label{eq:BoS-rate}
\end{eqnarray}

\item[(b)] If the system is stable, i.e., (\ref{eq:BoS-rate}) is satisfied, the steady state probabilities can be computed by the following iterative process:
\begin{eqnarray*}
\hspace{-.5in}&&\pi_0 = \frac{1-\eta}{(1-\eta)\sum_{l=0}^{2r-2}a_l + \frac{\lambda a_{2r-2}}{2r\mu} + a_{2r-1}}, \\
\hspace{-.5in}&& \pi_l=a_l \pi_0, \,\forall\, l\in\{1, ..., 2r-1\}, \\ 
\hspace{-.5in}&& \pi_{2r+2m-1}  = \frac{\lambda}{2r\mu} (\pi_{2r+2m-2}^{p} + \pi_{2r+2m-2}^g + \pi_{2r+2m-3}), \\
\hspace{-.5in}&&\pi_{2r}^{p} = \frac{1}{\gamma_p} \big[\beta_p (\pi_{2r-1} +\pi_{2r-2}) + \lambda \pi_{2r-2}\big] \\ 
\hspace{-.5in}&&\pi_{2r+2m}^{p}= \frac{1}{\gamma_p}  \big[  \beta_p (\pi_{2r+2m-2}^{p} + \pi_{2r+2m-2}^{g} + \pi_{2r +2m-1}) \\
\hspace{-.5in} && \qquad\quad+\lambda\pi_{2r+2m-2}^{p} -(2r-1)\lambda \pi_{2r+2m-2}^{g} \big], \,\,\forall\,\,m\geq1. 
\end{eqnarray*}
$\pi_{2r}^{g}$ and $\pi_{2r+2m}^{g}$ can be similarly computed by replacing $\gamma_p, \beta_p$ with $\gamma_g, \beta_g$.  $\Diamond$

\end{enumerate}
\end{theorem}
\begin{proof}
See Appendix A.
\end{proof}
Note that a system with $2r$ servers should ideally support any arrival rate $0\leq\lambda<r\mu$. 
%
However, we see from equation \eqref{eq:BoS-rate} of Theorem
\ref{theorem:BoS-per} that, under \tsf{BoS}, the there is a loss in throughput
of order $O(\frac{1}{r^2})$. 
This loss quickly goes to zero as $r$ increases.
Thus, \tsf{BoS} indeed ensures high throughput of the coded system even for
moderate values of $r$. 
%
%
Part (b) of Theorem \ref{theorem:BoS-per} then provides
an efficient way for analyzing the system.
We also note that our results are not asymptotic and can be applied to systems of small sizes. 

Using the approach in Theorem  \ref{theorem:BoS-per}, we can compute the average packet request delay in the system using the equation (\ref{eq:little-law}).
We then compute the delay gain of coding by: 
\begin{eqnarray}
\text{Delay Gain} = (d^{\text{uncoded}}_{\text{packet}}-d^{\text{coded}}_{\text{packet}})/d^{\text{uncoded}}_{\text{packet}}.
\end{eqnarray} 
Fig. \ref{fig:delay-gain} shows the delay gain for different values of $r$. We
see that the gain is significant even for small $r$, e.g., gain is $13\%$ for
$r=4$, and can be up to $17\%$ when $r=10$.
Finally, we emphasize that Fig. \ref{fig:delay-gain} is obtained \emph{analytically}, computed using results in Theorem  \ref{theorem:BoS-per}. 
We note that in Fig. \ref{fig:delay-gain}, for every $r$ value, delay gain is only plotted for arrivals that are within the capacity region of the system, i.e., (\ref{eq:BoS-rate}). 

\begin{figure}[cht]
\vspace{-.1in}
\centering
\includegraphics[width=3.3in, height=2in]{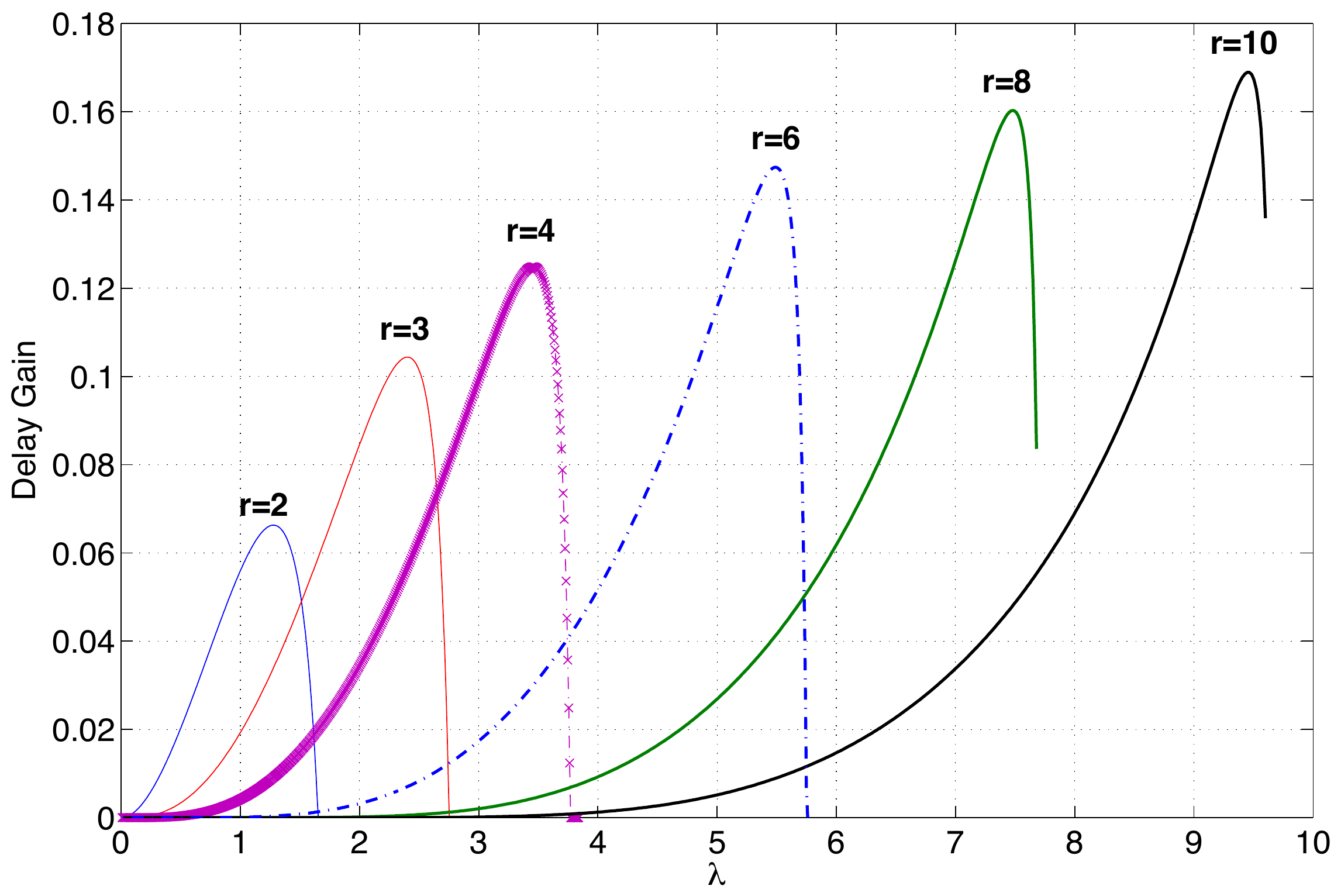}
\vspace{-.15in}
\caption{ Delay reduction of \tsf{BoS} over the replication scheme. When $r=4$, the delay gain is around $13\%$, and when $r=10$, the gain is $17\%$. 
The plots are generated using analytical results in Theorem  \ref{theorem:BoS-per}. 
The reason the gain decreases as the arrival rate increases is because \tsf{BoS} has a slightly smaller capacity region. Thus, if the rate is very close to the capacity region of \tsf{BoS}, the delay under \tsf{BoS} will be larger than that under the replication scheme. However, this happens only when the rate is very close to the capacity boundary. For most of the rates, codes achieve a significant delay reduction. 
%
}
\label{fig:delay-gain}
\vspace{-.1in}
\end{figure}

\vspace{-.2in}
\subsection{Average request delay}
Notice that the above analysis allows us to derive the average \emph{packet
request} delay. However, in practice, we care more about the average \emph{request}
delay, which is defined as the average time it takes for both packet 
requests of a request to get served. Our following lemma shows that under \tsf{BoS}, the average request
delay is roughly equal to the packet delay. This is a very desired feature not
possessed by the uncoded system.

\begin{lemma} \label{lemma:per-user-per-req}
Let $d^{\text{coded}}_{\text{req}}$ and $d^{\text{coded}}_{\text{packet}}$ be the average request delay and average packet request delay under \tsf{BoS}, then:
\begin{eqnarray}
d^{\text{coded}}_{\text{req}} &=& d^{\text{coded}}_{\text{packet}} +  \frac{2r-1}{2r-2} \frac{1}{2\mu} \label{eq:avg-user-request} \\
&&\quad  - \frac{1}{(2r-2)(2r-1)2r\mu}  - \frac{1}{2(2r-1)\mu}.\Diamond \nonumber
\end{eqnarray}
\end{lemma}
\begin{proof}
See Appendix B. 
\end{proof}
We see from the lemma that as number of servers $2r$ gets large, the difference between the packet delay and the request delay roughly equals $\frac{1}{2\mu}$.
This is exactly the difference between the average  of two requests service times and the maximum of them.
Hence, it will appear under any scheduling policies regardless of using coding or not.
Therefore, Lemma \ref{lemma:per-user-per-req} shows that \tsf{BoS} ensures that the average packet latency is almost equal to the average request latency. 
This is a very important feature of the \tsf{BoS} algorithm.

\section{Conclusion}\label{section:conclusion}
In this paper, we pose a fundamental question of the role of codes in improving
the latency of content retrieval in storage systems. The interplay between
coding and queueing delay is of complex nature. As a first step to make
progress on this complex problem we propose and analyze a simplified setting of
a single content divided into two parts and served by multiple servers. We see 
that even in this simplified setting the exact analysis of queueing delay for
the systems using codes is intractable. As a result we provide a sub-optimal
scheduling algorithm called Blocking-one Scheduling (\tsf{BoS}) that allows us to
theoretically quantify the gains in latency achieved by coded system as
compared to a system that uses replication. The methodology we developed in
this paper is applicable to a more general setting that allows splitting the
content into more than $2$ parts. Further generalizations of our work a) extending our
scheduling algorithm to incorporate more than one blocking to improve the
analytical gains in latency, and b) considering scenario of serving multiple
contents, is part of our future work.

\section*{Appendix A -- Proof of Theorem \ref{theorem:BoS-per}}
We now proof Theorem \ref{theorem:BoS-per} by analyzing the Markov chain  using a set of carefully chosen global balance equations, which we call ``cuts.'' Our approach is to first compute $\pi_0$. Then starting from $\pi_0$, we iteratively compute all the other probabilities. 

\begin{proof} (Theorem \ref{theorem:BoS-per})
First consider the sets that contain the $(2r+2m, p)$ and $(2r+2m, g)$ states, i.e., the Type $3$ cuts in Fig. \ref{fig:mc-system}. 
We note that if the system is stable, then the total transition rate going out from any set of states must be equal to the total rate going into them. Thus, 
we first have the following equation for the states $(2r, p)$ and $(2r, g)$: 
\begin{eqnarray}
\hspace{-.4in}&& \pi_{2r}^{p}(\lambda+2r\mu) + \pi_{2r}^{g}(\lambda+(2r-1)\mu)\label{eq:red-cuts1} \\
\hspace{-.4in}&=& \pi_{2r-2}\lambda + \pi_{2r+1}2r\mu. \nonumber
\end{eqnarray}
Then for states $(2r+2m, p)$ and $(2r+2m, g)$ with $m\geq1$, we have: 
\begin{eqnarray}
\hspace{-.4in}&&\pi_{2r+2m}^{p}(\lambda+2r\mu) + \pi_{2r+2m}^{g}(\lambda+(2r-1)\mu)\label{eq:red-cuts2} \\
\hspace{-.4in}&=& \lambda(\pi^{p}_{2r+2m-2}+\pi^{g}_{2r+2m-2}) + \pi_{2r+2m+1}2r\mu. \nonumber
\end{eqnarray}
Summing (\ref{eq:red-cuts1}) and (\ref{eq:red-cuts2}) over $m=1, 2, ...$, we get: 
\begin{eqnarray}
\hspace{-.4in}&&2r\mu\sum_{m=0}^{\infty} \pi_{2r+2m}^{p} + (2r-1)\mu\sum_{m=0}^{\infty}\pi_{2r+2m}^{g} \nonumber\\
\hspace{-.4in}&&= \lambda\pi_{2r-2} +2r\mu\sum_{m=0}^{\infty} \pi_{2r+2m+1}. \label{eq:sum-2km}
\end{eqnarray}
Now consider the diagonal cuts, i.e., the Type $2$ cuts, starting from states $(2r, p)$ and $(2r, g)$. We have: 
\begin{eqnarray}
\hspace{-.4in}&&(\pi_{2r-2}+\pi_{2r-1})\lambda   =\pi_{2r}^{p}2r\mu+\pi_{2r}^{g} (2r-1)\mu,\label{eq:blue-cuts}\\
\hspace{-.4in}&&(\pi_{2r+2m}^{p}+\pi_{2r+2m}^{g}+\pi_{2r+2m+1})\lambda \label{eq:blue-cuts2}\\
\hspace{-.4in}&&\qquad\qquad = \pi_{2r+2m+2}^{p}2r\mu +  \pi_{2r+2m+2}^{g}(2r-1)\mu,\,\, m\geq0. \nonumber
\end{eqnarray}
Summing (\ref{eq:blue-cuts}) and (\ref{eq:blue-cuts2}) over $m=0, 1, ...$, we get:
\begin{eqnarray}
2r\mu\sum_{m=0}^{\infty} \pi_{2r+2m}^{p} + (2r-1)\mu\sum_{m=0}^{\infty}\pi_{2r+2m}^{g} \nonumber\\
= \lambda\big[1-\sum_{l=0}^{2r-3}\pi_l\big]. \label{eq:sum-2km-2}
\end{eqnarray}
Using (\ref{eq:sum-2km}) and (\ref{eq:sum-2km-2}), we thus obtain: 
\begin{eqnarray}
2r\mu\sum_{m=0}^{\infty} \pi_{2r+2m+1}  = \lambda\big[1-\sum_{l=0}^{2r-2}\pi_l\big]. \label{eq:sum-odd}
\end{eqnarray}
We now try to first find $\pi_0$. Consider only the ``p'' states, we get:
\begin{eqnarray}
\hspace{-.5in}&& \quad\,\,\,\,\pi_{2r}^{p}(\lambda+2r\mu)  = \pi_{2r-2} \lambda +\pi_{2r+1} (2r-1)\mu,\label{eq:a-cut}\\
\hspace{-.5in}&& \pi_{2r+2m}^{p}(\lambda+2r\mu)  = \pi_{2r+2m-2}^{p} \lambda +\pi_{2r+2m+1} (2r-1)\mu. \label{eq:a-cut2} 
\end{eqnarray}
Summing (\ref{eq:a-cut}) and (\ref{eq:a-cut2}) over $m\geq1$,  we obtain:
\begin{eqnarray}
\hspace{-.5in}&&2r\mu\sum_{m=0}^{\infty}\pi_{2r+2m}^{p}= \lambda\pi_{2r-2} + (2r-1)\mu\sum_{m=0}^{\infty}\pi_{2r+2m+1}. \label{eq:2kma-2mkp1}
\end{eqnarray}
Similarly, we can look at the ``g'' states and get: 
\begin{eqnarray}
\hspace{-.5in}&&\quad\,\,\,\,\pi_{2r}^{g}(\lambda+(2r-1)\mu)  =  \pi_{2r+1}\mu,\label{eq:b-cut}\\
\hspace{-.5in}&&\pi_{2r+2m}^{g}(\lambda+(2r-1)\mu)  = \pi_{2r+2m-2}^{g} \lambda +\pi_{2r+2m+1} \mu. \label{eq:b-cut2} 
\end{eqnarray}
Summing these up, we get:
\begin{eqnarray}
(2r-1) \mu\sum_{m=0}^{\infty}\pi_{2r+2m}^{g} =\mu\sum_{m=0}^{\infty}\pi_{2r+2m+1}. \label{eq:2kmb-2mkp1}
\end{eqnarray}
Using (\ref{eq:sum-odd}), (\ref{eq:2kma-2mkp1}) and (\ref{eq:2kmb-2mkp1}), we obtain:
\begin{eqnarray*}
\hspace{-.4in}&& \sum_{m=0}^{\infty} \big[ \pi_{2r+2m}^{p} +  \pi_{2r+2m}^{g} +  \pi_{2r+2m+1}   \big] 
\\
\hspace{-.4in}&& = \bigg[  \frac{\lambda}{2r\mu} +\frac{\lambda(2r-1)\mu}{(2r\mu)^2} + \frac{\lambda\mu}{(2r-1)\mu2r\mu} \bigg]\big[1-\sum_{l=0}^{2r-2}\pi_l\big] \\
\hspace{-.4in}&&\quad\quad+ \frac{\lambda}{2r\mu}\pi_{2r-2}. 
\end{eqnarray*}
Therefore, we get: 
\begin{eqnarray}
\hspace{-.4in}&& \bigg[  \frac{\lambda}{2r\mu} +\frac{\lambda(2r-1)\mu}{(2r\mu)^2}  + \frac{\lambda\mu}{(2r-1)\mu2r\mu} \bigg]\big[1-\sum_{l=0}^{2r-2}\pi_l\big]  \label{eq:starting-pis}\\
\hspace{-.4in}&&\qquad\qquad\qquad\qquad+ \frac{\lambda}{2r\mu}\pi_{2r-2}  =\big[1-\sum_{l=0}^{2r-2}\pi_l\big] - \pi_{2r-1}. \nonumber
\end{eqnarray}
We see that (\ref{eq:starting-pis}) provides one equation in terms of only $\pi_0, ..., \pi_{2r-1}$. Below we show that all the probabilities $\pi_1, ..., \pi_{2r-1}$ can be expressed in terms of $\pi_0$. In this case (\ref{eq:starting-pis}) will allow us to compute  $\pi_0$ exactly. This in turn enables us to compute $\pi_1, ..., \pi_{2r-1}$. To do so, we first consider the type $1$ and type $2$ cuts shown in Fig. \ref{fig:mc-system} to get: 
\begin{eqnarray}
\hspace{-.5in}&&\quad\,\,\,\pi_1=\frac{\lambda}{\mu}\pi_0,\label{eq:starting-p0-1}\\
\hspace{-.5in}&&\quad\pi_{2i} =\frac{\lambda}{2i\mu}(\pi_{2i-2}+\pi_{2i-1}), \,\,\,\forall\,\, i\in\{1, ..., r-1\},  \label{eq:starting-p0-2}\\ 
\hspace{-.5in}&&\pi_{2i+1} = \frac{\lambda}{ (2i+1)\mu}(\pi_{2i-1}+\pi_{2i}),\,\forall\, i\in\{1, ..., r-1\}. \label{eq:starting-p0-3}
\end{eqnarray}
Using (\ref{eq:starting-p0-1}), (\ref{eq:starting-p0-2}) and  (\ref{eq:starting-p0-3}), one can obtain: 
\begin{eqnarray}
\pi_l &=& a_l \pi_0, \,\,\,\forall\,\, l\in\{0, ..., 2r-1\}, \label{eq:gen-pi-p0}
\end{eqnarray}
where $\{a_l, l=0, ..., 2r-1\}$ are defined as in (\ref{eq:a-def}). 
%
%
Plugging (\ref{eq:gen-pi-p0}) back into (\ref{eq:starting-pis}) and denote $\eta\triangleq \frac{\lambda}{2r\mu} +\frac{\lambda(2r-1)\mu}{(2r\mu)^2} + \frac{\lambda\mu}{(2r-1)\mu2r\mu}$, we have: 
\begin{eqnarray*}
1- \sum_{l=0}^{2r-2} a_l\pi_0  =\eta -\eta\sum_{l=0}^{2r-2}a_l\pi_0 + \frac{\lambda}{2r\mu}a_{2r-2}\pi_{0} +a_{2r-1}\pi_{0}. 
\end{eqnarray*}
Therefore: 
\begin{eqnarray} 
\pi_0 = \frac{1-\eta}{(1-\eta)\sum_{l=0}^{2r-2}a_l + \frac{\lambda a_{2r-2}}{2r\mu} + a_{2r-1}}. 
\end{eqnarray}
It is not difficult to verify that $\pi_0$ is a valid probability if $1-\eta>0$, i.e., 
\begin{eqnarray}
\frac{\lambda}{2r\mu} +\frac{\lambda(2r-1)\mu}{(2r\mu)^2} + \frac{\lambda\mu}{(2r-1)\mu2r\mu} < 1.  
\end{eqnarray}
This implies that the supportable rate is:
\begin{eqnarray}
\lambda < r\mu \frac{1-\frac{1}{2r}}{1 - \frac{1}{2r} + \frac{1}{8r^2}} = r\mu\bigg( 1-\frac{1}{8r^2-4r+1} \bigg). 
\end{eqnarray}
Here $r\mu$ is the total rate the system can ever support. Hence, we see that \tsf{BoS} only lose a fraction $\frac{1}{ 8r^2 - 4r + 1 }$. This proves Part (a). 
 
To prove Part (b), we first see that one can now use (\ref{eq:gen-pi-p0}) to compute $\pi_0, ..., \pi_{2r-1}$. 
To compute $\pi_{2r+2m}^{p}, \pi_{2r+2m}^{g}$, for all $m\geq0$, we start from $m=0$. We have from  (\ref{eq:blue-cuts}) that: 
\begin{eqnarray}
\pi_{2r}^{p} 2r\mu + \pi_{2r}^{g} (2r-1) \mu = \lambda (\pi_{2r-1} + \pi_{2r-2}).\label{eq:p2k-1}
\end{eqnarray}
Now if we look at the state $(2r, p)$ and $(2r, g)$ separately, we get:
\begin{eqnarray}
\hspace{-.2in}\pi_{2r}^{p} (\lambda+2r\mu) &=& \pi_{2r-2}\lambda + \pi_{2r+1} (2r-1)\mu,\label{eq:pi2ka}\\
\hspace{-.2in}\pi_{2r}^{g} (\lambda + (2r-1)\mu) &=& \pi_{2r+1}\mu. \label{eq:pi2kb}
\end{eqnarray}
Canceling the term $ \pi_{2r+1}$ in (\ref{eq:pi2ka}) and (\ref{eq:pi2kb}), we get that: 
\begin{eqnarray}
\hspace{-.05in}\pi_{2r}^{p} (\lambda+2r\mu)  - \pi_{2r}^{g} (2r-1)(\lambda+(2r-1)\mu) = \lambda \pi_{2r-2}. \label{eq:p2k-2}
\end{eqnarray}
With (\ref{eq:p2k-1}) and (\ref{eq:p2k-2}), we can now compute $\pi_{2r}^{p}, \pi_{2r}^{g}$. To make the expressions more concise, we define:
\begin{eqnarray*}
\hspace{-.3in}&&\gamma_p \triangleq  \frac{2r\mu(\lambda+(2r-1)\mu)}{\mu}  + (\lambda+2r\mu), \\
\hspace{-.3in}&&\gamma_g \triangleq  \frac{-(2r-1)\mu(\lambda+2r\mu)}{2r\mu}  - (2r-1) (\lambda +(2r-1)\mu),\\
\hspace{-.3in}&&\beta_p \triangleq \frac{\lambda (\lambda + (2r-1)\mu)}{\mu}, \quad \beta_g \triangleq  \frac{-\lambda(\lambda+2r\mu)}{2r\mu}. 
\end{eqnarray*}
Then we get: 
\begin{eqnarray}
\pi_{2r}^{p} &=& \frac{1}{\gamma_p} \big[\beta_p (\pi_{2r-1} +\pi_{2r-2}) + \lambda \pi_{2r-2}\big],\\
\pi_{2r}^{g}  &=&\frac{1}{\gamma_g} \big[ \beta_g(\pi_{2r-1} +\pi_{2r-2}) +\lambda\pi_{2r-2}  \big].
\end{eqnarray}
Now for all the states $(2r+2m, p)$ and $(2r+2m, g)$  with $m\geq1$, using (\ref{eq:blue-cuts2}), (\ref{eq:a-cut2}) and (\ref{eq:b-cut2}), we get: 
%
\begin{eqnarray*}
&& 2r\mu \pi_{2r+2m}^{p} + (2r-1)\mu\pi_{2r+2m}^{g} \\
&&\qquad\qquad= \lambda (\pi_{2r+2m-2}^{p} + \pi_{2r+2m-2}^{g} + \pi_{2r +2m-1}),\\
&& (\lambda + 2r\mu) \pi_{2r+2m}^{p} - (2r-1) (\lambda + (2r-1)\mu) \pi_{2r+2m}^{g}\\
&&\qquad\qquad= \lambda \pi_{2r+2m-2}^{p} - (2r-1)\lambda \pi_{2r+2m-2}^{g}. 
 \end{eqnarray*}
We can thus obtain the following equations for all states $2r+2m, m\geq1$: \footnote{It can be verified that both $\pi_{2r+2m}^{p}$ and $\pi_{2r+2m}^{g}$ are both positive. Thus they are valid probabilities.}
\begin{eqnarray*}
\pi_{2r+2m}^{p} &=& \frac{1}{\gamma_p}  \big[  \beta_p (\pi_{2r+2m-2}^{p} + \pi_{2r+2m-2}^{g} + \pi_{2r +2m-1}) \\
&& \qquad +\lambda\pi_{2r+2m-2}^{p} -(2r-1)\lambda \pi_{2r+2m-2}^{g} \big], \\
\pi_{2r+2m}^{g} &=&  \frac{1}{\gamma_g}  \big[  \beta_g (\pi_{2r+2m-2}^{p} + \pi_{2r+2m-2}^{g} + \pi_{2r +2m-1}) \\
&& \qquad +\lambda\pi_{2r+2m-2}^{p} -(2r-1)\lambda \pi_{2r+2m-2}^{g} \big]. 
\end{eqnarray*}
Then, the probabilities $\pi_{2r+2m-1}$ with $m\geq1$ can be computed using type $1$ cuts, i.e., 
\begin{eqnarray}
2r\mu\pi_{2r+2m-1} = \lambda (\pi_{2r+2m-2}^p+\pi_{2r+2m-2}^g + \pi_{2r+2m-3}). 
\end{eqnarray}
%
With all the above results, the average packet request delay in the system can be computed as: 
\begin{eqnarray*}
d_{\text{packet}}^{\text{coded}} = \frac{1}{2\lambda}\bigg[\sum_{l=1}^{2r-1}l\pi_l + \sum_{m\geq0} (2r+2m+1) \pi_{2r+2m+1} \\
+ 
\sum_{m\geq0} (2r+2m)(\pi_{2r+2m}^{p} +\pi_{2r+2m}^{g})\bigg].
\end{eqnarray*}
This completes the proof of Part (b). 
\end{proof}

\section*{Appendix B -- Proof of Lemma  \ref{lemma:per-user-per-req}}
Here we prove Lemma  \ref{lemma:per-user-per-req}. 
\begin{proof} (Lemma \ref{lemma:per-user-per-req}) 
Consider any request $u$ that enters and departs from the system.  
Let $w_{u1}$ and $w_{u2}$  be the waiting times of its first and second packet requests  in the queue. Then let $s_{u1}$ and $s_{u2}$ be the service times of the two packet requests. 

We now derive a relationship between $w_{u1}$ and $w_{u2}$. According to \tsf{BoS}, packet request $1$ always goes before packet request $2$, thus $w_{u1}\leq w_{u2}$. Now suppose packet request $1$ goes to a storage unit $j$ at a time $t$. Then, the extra waiting time of packet request $2$ in the queue is exactly the time it takes for any of the other $2r-1$ storage units to be free. 
By the exponential service time nature, this time is exponentially distributed with rate $(2r-1)\mu$, i.e., 
\begin{eqnarray}
w_{u2} = w_{u1} + \tau_u, \quad\text{where}\,\,\, \tau_u\sim \exp((2r-1)\mu).
\end{eqnarray}
Now let $t_{u1}$ and $t_{u2}$ be the total times packet request $1$ and $2$ stay in the system, and let $t_u$ be the time the request spends in the system, we have: 
\begin{eqnarray}
t_{u1}  &=& w_{u1} + s_{u1}, \\
t_{u2}  &=& w_{u2} + \tau_u + s_{u2},\\
t_u&=&w_{u1} + \max(s_{u1}, s_{u2}+\tau_u). 
\end{eqnarray}
Denote $Z = \max(s_{u1}, s_{u2}+\tau_u)$. It can be verified that: 
\begin{eqnarray*}
\hspace{-.3in}&&f_Z(z) = \frac{2r-1}{2r-2}\mu e^{-\mu z} - \frac{(2r-1)\mu}{2r-2} e^{-(2r-1)\mu z} \\
\hspace{-.3in}&&\qquad\qquad + \mu e^{-\mu z} - \frac{2r-1}{2r-2} 2\mu e^{-2\mu z} + \frac{1}{2r-2}\mu e^{-2r\mu z}, 
\end{eqnarray*}
and: 
\begin{eqnarray}
\expect{Z} = \frac{1}{\mu} + \frac{2r-1}{2r-2} \frac{1}{2\mu}  - \frac{1}{(2r-2)(2r-1)2r\mu}. \end{eqnarray}
It thus follows that the difference between the average request delay and the average packet request delay  under \tsf{BoS} is given by: 
\begin{eqnarray*}
\hspace{-.3in}&&d^{\text{coded}}_{\text{req}}-d^{\text{coded}}_{\text{packet}}  =\expect{t_u} - \frac{1}{2} \big(\expect{t_{u1}} +\expect{t_{u2}} \big) \\
\hspace{-.3in}&&\qquad\quad\,\,\,  =  \frac{2r-1}{2r-2} \frac{1}{2\mu}   - \frac{1}{(2r-2)(2r-1)2r\mu}  - \frac{1}{2(2r-1)\mu}. 
\end{eqnarray*}
This completes the proof. 
\end{proof}

\pagestyle{empty}
\bibliographystyle{plain.bst}
\bibliography{coding-delay-refs}
\end{document}